\documentclass{amsart}
\usepackage[dvips]{graphicx}
\usepackage{amsmath}
\usepackage{amsfonts}
\usepackage{amssymb}
\usepackage{eufrak}
\usepackage[all]{xy}

\theoremstyle{plain}
\newtheorem{theorem}{Theorem}[section]
\newtheorem{lemma}[theorem]{Lemma} 
\newtheorem{prop}[theorem]{Proposition}
\newtheorem{cor}[theorem]{Corollary}
\newtheorem{defin}[theorem]{Definition}

\newcommand{\cF}{\mathcal{F}}
\newcommand{\cL}{\mathcal{L}}
\newcommand{\ro}{\varrho}

\DeclareMathOperator{\Core}{\mathsf{Core}}

\DeclareMathOperator*{\LW}{\bigg\rmoustache_{\cL}}

\author{Attila Egri-Nagy and Chrystopher L. Nehaniv} 
\address{School of Computer Science\\
        University of Hertfordshire\\
        College Lane\\
        Hatfield, Herts\\
        UK}
\email{\{A.Egri-Nagy,\ C.L.Nehaniv\}@herts.ac.uk}


\title[Subgroup Chains and Lagrange Coordinatizations ]
      {Subgroup Chains and Lagrange Coordinatizations of Finite Permutation Groups}

\begin{document}

\begin{abstract}
We give a general constructive proof for hierarchical coordinatizations (Lagrange Decompositions) of permutation groups. The generalization originates from the investigation of how the subgroup chains of finite permutation groups yield different coordinate systems. The study is motivated by the practical needs and the verification of an existing computational implementation. Large scale machine calculated examples are also presented.
\end{abstract}

\maketitle

\section{Introduction}

We consider coordinatizations of finite permutation groups, i.e.\ hierarchical decompositions into subwreath products. Ultimately, we would like to use these coordinate systems as cognitive tools for understanding and manipulating processes describable by permutation groups. Prominent example is our positional number notation system, a coordinate system built from copies  of ${\mathbb Z}_{10}$, modulo 10 counters. However, for preparing real-world applications we need to investigate the nature of these coordinate systems. 

There are many different attributes of a hierarchical decomposition describing its dimensions, complexity of the components and their connection network. It turns out that these are all determined by the subgroup chain that underpins the decomposition,  but the chain itself is not the right form for enabling easy calculation in the decomposition.
The Jordan-H\"{o}lder Theorem gives  decompositions but not a calculus\footnote{In computer science terms, by coordinatization we put a user interface on the group structure.}.  Here we study how the attributes of the chains can be translated into the attributes of the coordinate systems. The outcome of this investigation is meant to be a mathematical toolbox for 'engineering' coordinatizations.

These coordinatizations use the  idea behind induction in representation theory (see e.g.\ \cite{AlperinBell}), so it traces back to Frobenius and it is also known as the Krasner-Kaloujnine embedding \cite{KrasnerKaloujnine}. All we need here is just standard group theory, namely the cosets, hence the name Lagrange Decomposition. Strictly speaking very little new mathematical results are presented here, however a different perspective, a new way of thinking is introduced: we actually build the coordinate systems with their dependency structure in an efficient way instead of only establishing embedding into the wreath product. For practical applications  and computer science this may be revolutionary. 

The constructive proof given here closely follows the computationally implemented algorithms \cite{sgpdec} for increasing usability and enabling verification of the software package.
\subsection{Notation and Terminology}

A \emph{subgroup chain} of group $G$ is a sequence of groups such that $G=G_1\geq\ldots\geq G_n$. If $G_n=\{1\}$ then the chain is \emph{total}. For reducing the notational burden we simply write $(G_1,\ldots,G_n)$ for the chain. A subgroup chain is \emph{subnormal} if $G_i\rhd G_{i+1}$ for all $1 \leq i < n$.

\noindent In addition to the usual \emph{permutation group} notation $(X,G)$ we also use $[X,G]$ denoting a \emph{group acting by permutations} when the action is not necessarily faithful.

\noindent The \emph{core} or \emph{normal interior} of subgroup $H$ in $G$ is 
$$\Core_G(H)=\bigcap_{g\in G}g^{-1}Hg,$$
which is the largest normal subgroup of $G$ contained in $H$. The subgroup $H$ is \emph{core-free} in $G$ if $\Core_G(H)=\{1\}$. See standard references \cite{RobinsonGroups,CameronPermGroups99}.

\section{Cascaded Structures Built from  Groups Acting by Permutations}
Here we describe a different way of thinking about wreath products. The emphasis is put on the connection network between the components of the product and on the substructures of the full wreath product. Also, this approach is more constructive, instead of establishing an embeddding into a wreath product, we would like to actually build a group hierarchically from simpler components. This is in accordance with the recent directions of group theory \cite{SautoyNewScientist2008,FindingMoonshine2008}.  
Clearly, the following construction is the same for permutation groups and transformation semigroups. 

Let $L=[X_1,C_1],\ldots,[X_n,C_n]$ be an ordered list of groups $C_i$ acting by permutations on sets $X_i$, calling $[X_1,C_1]$ the top and $[X_n,C_n]$ the bottom level component\footnote{The ordering is due to the constraints of a software implementation, as in computer algebra system lists are usually indexed by starting from 1. This partially clashes with the mathematical canon, but as we would like to describe and verify our algorithms, we simply have no choice.}.
Let $F_i$, $i\in\{1,\ldots, n\}$, each be a family of functions from $X_1\times\ldots\times X_{i-1}$ to $C_i$. Such a function $f_i\in F_i$, called a \emph{dependency function}, determines the action on the $i$th level depending on the states of the levels above.
Then a \emph{cascaded structure} built from $L$ is any  group acting by permutations of the form
$$[X_1\times\ldots\times X_n, \cF\subseteq F_1\times\ldots\times F_n] $$
\noindent denoted by 
$$ [X_1,C_1]\wr_{\cF}\ldots\wr_{\cF} [X_n,C_n].$$
\noindent The action is defined by
\begin{equation}
\label{eq:cascact}
(x_1,\ldots, x_n)\cdot (f_1, \ldots, f_n) = (y_1,\ldots, y_n)
\end{equation}
\noindent where
\begin{align*}
y_1&=f_1() \text{  constant function taking value in}\ C_1,\\
y_i&=x_i\cdot f_i(x_1,\ldots, x_{i-1}),\   x_i \in X_i, f_i\in F_i, 2\leq i \leq n.
\end{align*}
 $\cF$ is called the \emph{dependency structure}, a \emph{system of dependencies}, or simply the \emph{`wiring'}.
\subsection{Wreath Product} If $F_i=C_i^{X_1\times\ldots\times X_{i-1}}$, i.e.\ the set of all functions from  $X_1\times\ldots\times X_{i-1}$ to $C_i$ and $ \cF= F_1\times\ldots\times F_n$, then we have the wreath product of the groups in $L$ denoted by $[X_1,C_1]\wr\ldots\wr [X_n,C_n]$. Thus cascaded structures are substructures of the wreath product. Except for a small set of components, the wreath product is a huge structure and it can become easily intractable computationally.

\noindent\textbf{Remark.} The direct products in the above constructions are set theoretic, and they are not equipped with multiplication. The multiplication within the cascaded structures is a complicated operation when considered componentwise, and it is given by (\ref{eq:cascact}) and function composition. 
\section{Lagrange Coordinates}

The coordinatization of the right regular representation of a permutation group is the easiest to describe, therefore we construct coordinates for $(G,G)$, then we proceed to other representations $(X,G)$ and show how the construction changes.

\begin{theorem}[Lagrange Coordinatization]
\label{thm:lagrangecoords}
 Let $G$ be a group and $(G_1, \ldots, G_n)$ be a total subgroup chain of $G$, then the permutation group $(G,G)$  admits the following coordinatization
$$(G,G)\twoheadleftarrow\LW_{1\leq i<n}[G_i/G_{i+1},G_i],$$
which is a bijection on states.\footnote{Note that the cascaded structure used in the coordinatization is not the full wreath product.}
\end{theorem}

 Note that unlike previous formulations (e.g.\ \cite[Ch.\ 1]{pdclnbookA}), we do not require the chain to be subnormal. As the components are not necessarily faithful, the coordinatization is a surmorphism.   First we show how to assign coordinate values to the elements of $G$ as states, then describe how to construct a set of dependencies for any $g\in G$ as a permutation (thus building $\cL$). These will serve as a constructive proof of Theorem \ref{thm:lagrangecoords}.

\subsection{Coordinatizing States}

For each consecutive pairs in the list we construct the set of right cosets $G_i/G_{i+1}$. These are the state sets of the components in the cascaded structure of Theorem \ref{thm:lagrangecoords}. As usual, we choose arbitrary but fixed representatives for the cosets and act on them instead of the cosets themselves. It is not absolutely necessary, but to make calculations shorter from now on, when possible we always choose the identity permutation to be a representative element. As multiplications in the group can end up anywhere within the cosets, we require to have an operation that takes any element to its coset representative: $g\mapsto \overline{g}$. However, the notation is a bit ambiguous as it needs to be clear from the context that in which set of cosets we take the representative element. Therefore if it is needed to avoid ambiguity, we index the bar along a chain, $\overset{\scriptscriptstyle{i}}{\overline{g}}$ meaning that it is a representative element of a coset in $G_i/G_{i+1}$, thus  $\overset{\scriptscriptstyle{i}}{\overline{g}}\in G_i$. 

 The following basic properties of cosets are stated in a lemma, as they  will be used often later on.

\begin{lemma}
\label{broadway}
Let $G$ be a group and $H<G$. Then we have the following for the right coset representatives of $G/H$. For any $g,k\in G$,
\begin{enumerate}
\item $\overline{g}=\overline{\overline{g}}$
\item $\overline{gk}=\overline{\overline{g}k}$
\end{enumerate}
\end{lemma}

\begin{proof} (1) is trivial. (2) By considering the action of $G$ on cosets of $H$, the statement is obvious: $H\overline{g}k=(H\overline{g})k=(Hg)k=H{gk}$.
\end{proof}

\begin{defin}The action of $G$ on coset representatives\footnote{Actually, all the following constructions and proofs can be described as acting on cosets, which would make the proofs easier. However, in a computational implementation we cannot calculate the images of potentially big cosets, but hit a representative element and correct if the resulting image is not a coset representative.} for $G/H$ is given by
$$\overline{g}*k=\overline{{\overline{g}k}}.$$

\end{defin}

\subsubsection{Raising Group Elements as States} A \emph{cascaded state} in the cascaded structure is a tuple  of coset representatives $(\overline{g_1}, \ldots, \overline{g_{n-1}})$ where $\overline{g_i}$ is the representative element in $G_i$ for a coset $G_{i+1}g_i$, $g_i\in G_i$.
Now we establish a mapping from the elements of $G$ to the cascaded states, called \emph{raising}, $\ro: G \rightarrow \prod_{1\leq i <  n}G_i/G_{i+1}$, and the inverse operation  $\phi=\ro^{-1}$ is called \emph{flattening}. Raising, $\ro: g\mapsto (\overline{g_1}, \ldots, \overline{g_{n-1}})$, is defined recursively and done in two stages. First we locate the permutations  describing the  action of $g$ in the subgroups.
\begin{defin}[Locating Permutations within Subgroups]
\label{def:locator}
Let $G$ be a group and $(G_1, \ldots, G_n)$ be a total subgroup chain of $G$ and $g\in G$. We define the map $g \mapsto (g_1,\ldots, g_{n-1})$ as 
\begin{eqnarray*}
g_1 &=& g\\
g_i&=& g_{i-1}\cdot(\overset{\scriptscriptstyle{i-1}}{\overline{g_{i-1}}})^{-1}, \ \ 1< i <n.
\end{eqnarray*}
\end{defin} 
\noindent So starting from the identity element (as the representative of the coset of a subgroup) we go to an element $g$ possibly ending up in another coset, where we go to the coset representative. This last step is projected back to the subgroup by taking the inverse of the representative element. In other words, the computation in a translate of the subgroup is expressed within the subgroup (Fig. \ref{fig:locating}). 

We need to show that these coordinate values are in the right subgroups, i.e.\ that $g_i\in G_i$, so that $\bar{g_i}$ is well-defined. 
\begin{lemma}
\label{lemma:rightaction}
Let $G$ be a group and $(G_1, \ldots, G_n)$ be a total subgroup chain of $G$. For $g\in G$ locate $(g_1,\ldots, g_{n-1})$ as in Definition \ref{def:locator},
then $g_{i}\in G_i$. 
\end{lemma}
\proof The statement is true for the top level, $g_1=g\in G=G_1$. Now inductively, given that $g_i\in G_i$, locating the next coordinate gives $g_{i+1}= g_i\overline{g_i}^{-1}$. Now let's consider the following set maps for the coset $G_{i+1}$ in $G_i/G_{i+1}$ given by right multiplication by fixed elements of $G$:
\begin{center}
\mbox{
\xymatrix@C=31pt@R=11pt{
G_{i+1}\ar@{->}[r]^{g_i} & G_{i+1}g_i\\
G_{i+1}g_{i}\ar@{->}[r]^{\overline{g_i}^{-1}}& G_{i+1} \\
G_{i+1}\ar@{->}[r]^{g_i\overline{g_i}^{-1}}& G_{i+1}
}}
\end{center} 
\noindent thus $g_i\overline{g_i}^{-1}\in G_{i+1}$. The composite map is trivial only if $g_i=\overline{g_i}$. 
\endproof
\begin{figure}
\includegraphics[width=.66\textwidth]{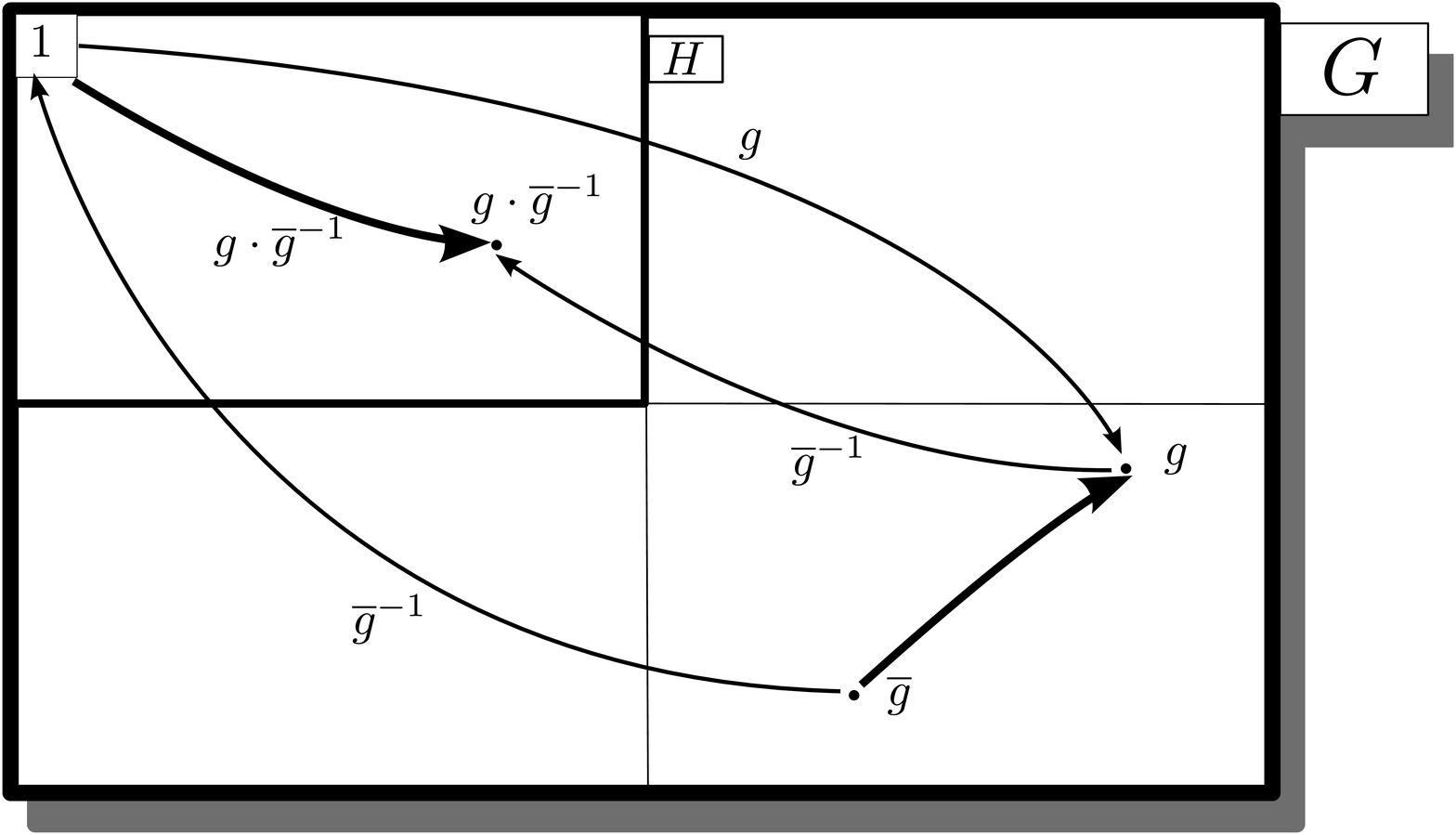}
\caption{Locating the permutation corresponding to $g\in G$ in $H<G$ with respect to the right cosets $G/H$.}
\label{fig:locating}
\end{figure}
\noindent For finishing raising, as the second step, we simply switch to the representative elements $g_i\mapsto\overline{g_{i}}$, in order to have valid coordinate values.

\subsubsection{Flattening states}\label{sect:flatstates} Now given $\ro_s(g)=(\bar{g_1},\ldots, \bar{g_n})$ we would like to find the element $g=\phi(\ro_s(g))\in G$. Flattening reveals the purpose of the recursive trickery above, $$\phi: (\overline{g_1},\ldots, \overline{g_{n-1}})\mapsto \overline{g_{n-1}}\cdot\overline{g_{n-2}}\cdots \overline{g_1},$$so we simply do the product of the coordinates bottom up.
Expanding $\overline{g_{n-1}}$ we get
$$\overset{\scriptscriptstyle{n-1}}{\overline{g_{n-2}\overset{\scriptscriptstyle{n-2}}{\overline{g_{n-2}}}{}^{-1}}}\cdot\overset{\scriptscriptstyle{n-2}}{\overline{g_{n-2}}}\cdot\overset{\scriptscriptstyle{n-3}}{\overline{g_{n-3}}}\cdots \overset{\scriptscriptstyle{1}}{\overline{g_1}}$$
but as on the deepest level we have the cosets of the trivial group (when decomposing along total chains) we can remove the top bar, since the cosets are singletons.
$$g_{n-2}\underbrace{\overline{g_{n-2}}^{-1}\cdot\overline{g_{n-2}}}_1\cdot\overline{g_{n-3}}\cdots \overline{g_1}.$$
The cancellation makes another one possible, like falling dominoes. Generally,
$$g_i\cdot \overline{g_{i-1}}\cdot\overline{g_{i-2}}\cdots = g_{i-1}\overline{g_{i-1}}^{-1}\cdot \overline{g_{i-1}}\cdot\overline{g_{i-2}}\cdots = g_{i-1}\cdot\overline{g_{i-2}}\cdots$$
where $i$ goes down to 1, leaving only $g_1$, which is by definition equals to $g$.  
Thus  $\phi(\ro_s(g))=g$, therefore the bijection is established.

For mathematical purposes we could have a much shorter way for proving the bijection. By Lagrange Theorem $|G|= |\prod_{1\leq i <  n}G_i/G_{i+1}|$ (see Appendix \ref{prop:lagrangesize}), and it is easy to show inductively that if two cascaded states map down to the same element, then they should have the same coordinates on all levels (as cosets containing the same group element are unique). However, in  a computational settings we need to actually calculate coordinates.


\subsection{Coordinatizing Permutations}

Similarly to states, we would like to raise group elements as permutations and flatten cascaded permutations. We can actually reuse the symbols $\ro$ and $\phi$, but to distinguish we use $\ro_s$ for raising states and $\ro_p$ for raising permutations.
Thus for $h\in G$ as permutation, raising gives a tuple of dependency functions $\ro_p(h)=(h_1,\ldots, h_{n-1})$, a member of $\cL$.
This set of dependencies is a quite complicated object, it is a labelled tree\footnote{ 'Acting on trees' seems to be the unifying idea of all concepts of cascaded structures. See \cite{monoidtreeactrhodes,monoidtreeactnehaniv}}. The arrows are labelled by the elements of the state sets of the components, and the nodes by the elements of the components (by the values of the dependency functions).
Due to this complexity of the object we cannot describe them explicitly (only in very simple cases). Instead, we define them recursively and give the values of dependency functions on concrete coordinates, i.e.\ on a path in the tree. We call these coordinate value permutations \emph{component actions}. Let $\ro_p{h}=(h_1,\ldots,h_{n-1})\in\cL$ and $(\overline{g_1},\ldots,\overline{g_{n-1}})$ is a coordinatized state, then the action is 
\begin{eqnarray} (\overset{\scriptscriptstyle{1}}{\overline{g_1}},\ldots,\overset{\scriptscriptstyle{n-1}}{\overline{g_{n-1}}}) * \ro_p(h)&=& (\overline{g_1}*h_1,\ldots,\overline{g_{n-1}}*h_{n-1})\\
&=&\big(\overset{\scriptscriptstyle{1}}{\overline{\overset{\scriptscriptstyle{1}}{\overline{g_1}}\cdot h_1}},\ldots,\overset{\scriptscriptstyle{n-1}}{\overline{\overset{\scriptscriptstyle{n-1}}{\overline{g_{n-1}}}\cdot h_{n-1}}}\big)
\nonumber
\label{eq:lagraction}
\end{eqnarray}
where the $h_i$'s are defined recursively by
\begin{eqnarray}
h_1&=&h\nonumber \\
h_i&=&\overset{\scriptscriptstyle{i-1}}{\overline{g_{i-1}}}h_{i-1}\Big(\overset{\scriptscriptstyle{i-1}}{\overline{\overset{\scriptscriptstyle{i-1}}{\overline{g_{i-1}}}h_{i-1}}}\Big)^{-1}.\label{eq:fudge}
\end{eqnarray}
\noindent The new notation $*$ for the action is introduced to distinguish it from the original group operation; the difference is mainly that we take the representative element after the multiplication in case of $*$. Note that $h_i$ really is a dependency function with arguments $(\overline{g_1},\ldots,\overline{g_{i-1}})$. The idea of Lemma \ref{lemma:rightaction} applies here as well, thus $\overline{g_{i-1}}h_{i-1}\overline{\overline{g_{i-1}}h_{i-1}}^{-1}\in G_{i}$. It is also clear, that locating permutations in the subgroups (Definition \ref{def:locator}) is a special case of these component actions, namely the ones we get when we apply $\ro_p(g), g\in G$ as permutation to the cascaded state consisting of the identities on each level (which is the cascaded state corresponding to the identity of $G$, by convention).

\begin{prop}
\label{prop:isomorphism}
$\ro_s(g\cdot h) = \ro_s(g)\cdot\ro_p(h)$.
\end{prop}
\proof

 For the top level, $i=1$, the statement is true, as  $\ro_s(gh)_1=\overline{gh}=\overline{\overline{g}h}=\ro_s(g)_1\cdot\ro_p(h)_1$, using Lemma \ref{broadway}.

 We proceed by induction, assuming that $\ro_s(gh)_i=\overline{\overline{g_i}\cdot h_i}=\overline{g_i\cdot h_i}$, by Definition \ref{def:locator} the next state coordinate in $\varrho_s(g)$ is
$$\overset{\scriptscriptstyle{i+1}}{\overline{g_{i+1}}}=\overset{\scriptscriptstyle{i+1}}{\overline{g_i\overset{\scriptscriptstyle{i}}{\overline{g_i}}{}^{-1}}}$$
By (\ref{eq:fudge}) the next component action of $\ro_p(h)$ on $\ro_s(g)$ is
$$ h_{i+1}=\overline{g_i}h_i\overline{\overline{g_i}h_i}^{-1}\in G_{i+1}$$  

\noindent Now, carrying out the component action by (\ref{eq:lagraction}) 
$$\overline{g_{i+1}} * h_{i+1} = \overset{\scriptscriptstyle{i+1}}{\overline{\overset{\scriptscriptstyle{i+1}}{\overline{g_i\overline{g_i}^{-1}}}\cdot  \overline{g_i}h_i\overline{\overline{g_i}h_i}^{-1}}}$$
by Lemma \ref{broadway}(2)

$$=\overline{g_i\underbrace{\overline{g_i}^{-1}\cdot  \overline{g_i}}_{1}h_i\overline{\overline{g_i}h_i}^{-1}}$$
$$=\overline{g_ih_i\overline{\overline{g_i}h_i}^{-1}}$$
after cancellation, applying  Lemma \ref{broadway}(2) again
$$=\overline{g_ih_i\overline{g_ih_i}^{-1}}$$
\noindent then by the induction assumption and using Definition \ref{def:locator}
$$=\overline{(gh)_i\overline{(gh)_i}^{-1}}=\ro_s(gh)_{i+1}.$$

This is an embedding since $\ro_s(1)\cdot \ro_p(g)=\ro_p(g)$ determines $g$, and is a bijection since $\phi$ is the inverse of $\ro_s$.

Also, for all $h,h'\in G$ we have that the action of $\ro_p(hh')$ is that
same as the action of $\ro_p(h)$ followed by that of $\ro_p(h')$:
Let $x=(g_1,\ldots,g_{n-1})$ be any state, and let $g=\phi(x)$. Then,
applying what we have just shown above, 
$x\cdot \ro_p(h)\cdot \ro_p(h') 
= \ro_s(g)\cdot \ro_p(h)\cdot\ro_p(h')
= \ro_s(gh)\cdot\ro_p(h')
= \ro_s(ghh')
=\ro_s(g)\cdot \ro_p(hh')$. Thus the actions of $\ro_p(hh)$ and $\ro_p(h)\ro_p(h')$ are equal on the set of all states. 
In particular, the $i^{\rm th}$ component actions are equal modulo the core of $G_{i+1}$ in $G_i$.  It is not hard to see that $\ro_p^{-1}$ is surjective onto $G$. Thus, we have a surjective mapping of actions of groups which is bijective on states.
\endproof
With this proposition we have established isomorphism between $G$ and its coordinate system based on a total subgroup chain, therefore we have proved the Lagrange Decomposition Theorem.

Notation:  $\mathfrak{L}\big(G\mid G>\ldots>\langle1\rangle\big)$ denotes  the Lagrange decomposition of $G$ based on the given chain.

\subsection{Obtaining Permutation Group Components}
In order to get permutation group components for Theorem \ref{thm:lagrangecoords} we need to make  the action $[G/H,G]$ faithful (for a consecutive pair $G>H$ in the chain ). If $G\rhd H$ then simply the factor group $G/H$ is the faithful action.
In the general case we act on $G/H$ by $G/\Core_G(H)$ instead, or shortly $G/_{\Core}H$. Algorithmically we calculate how the generators of $G$ act on $G/H$, thus we get a new generating set. Then we remove duplicated generators.  

This  way we have a more precise version of Theorem \ref{thm:lagrangecoords} that establishes isomorphism:

\begin{cor}
 Let $G$ be a group and $(G_1, \ldots, G_n)$ be a total subgroup chain of $G$, then the permutation group $(G,G)$  admits the following coordinatization
$$(G,G)\cong\LW_{1\leq i<n}(G_i/G_{i+1},G_i/_{\rm Core}\,G_{i+1}).$$
\end{cor}
\proof
The statement is immediate from the proof of Theorem \ref{thm:lagrangecoords} and from the fact that making the action faithful is equivalent to factoring by the core: two elements of $G_i$ are equivalent modulo the core iff they
act the same on cosets of $G_i/G_{i+1}$.
\endproof
\subsection{Basic Attributes of Coordinatizations}

The \emph{length} of the coordinatization is the number of dimensions, the number of hierarchical levels of the decomposition. If the underlying chain has $n$ members, then we have $n-1$ components, thus the length is $n-1$. The intuition is that longer decompositions yield simpler components, where simpler could mean reduced number of symmetries or states (or both). As a degenerate case the trivial coordinatization of group $G$ based on the chain $G>\langle 1\rangle$ is $G$ itself.

The \emph{width} of a  component  is the number of points it acts on, the number of coordinate values on that level. In Lagrange Coordinatization it is the index $G_i:G_{i+1}$.

\subsection{Coordinatizing Transitive Actions}

We saw that coordinatizing according to a total chain gives the right regular action. But we also need coordinatizations of acting on smaller sets as well, i.e.\ we would like to build a cascaded structure isomorphic to $(X,G)$ where $|X|<|G|$ and $G$ acts on $X$ transitively.

What are those smaller actions? Though it is quite a basic question, the answer is very rarely included in standard group theory textbooks. From \cite{CameronPermGroups99}: given a group $G$, the isomorphic transitive permutation groups are classified by the conjugacy classes of core-free subgroups of $G$. If $H$ is core-free in $G$ then $(G/H,G)$ is a permutation group.
In order build a cascaded structure isomorphic to this action, we need to   cut the total chain at $H$. The only construction that relies on the totality of the subgroup chain is flattening the states (Section \ref{sect:flatstates}). There in order to remove the $n-1^{\rm th}$ bar we needed the trivial group, but since here we are only interested in the action on the cosets of $H$ not on their elements, the removal of the last bar is possible in this more general case, using the fact that $Hg=H\overline{g}$ for the  cosets of $H$.  
Thus we have

\begin{theorem}[Lagrange Coordinatization for Transitive Actions]
\label{thm:lagrangecoordsonGX}
Let $G$ act on $X$ transitively.  Let $G=G_1 > \ldots> G_n=H$ be a subgroup chain for $G$, where $H$ is the stabilizer of some element of $X$. 
Then  $[X,G]$  admits the following coordinatization
$$[X,G]\twoheadleftarrow\LW_{1\leq i<n}[G_i/G_{i+1},G_i],$$
which is a bijection on states.\\
If in addition $(X,G)$ is a permutation group, then
$(X,G)$  admits the following coordinatization
$$(X,G)\cong\LW_{1\leq i<n}(G_i/G_{i+1},G_i/_{\rm Core}\,G_{i+1}).$$
\end{theorem}

\section{Example Coordinatizations}
\subsection{Rotational Symmetries of the Tetrahedron}
The rotation group of the tetrahedron is the alternating group $A_4$. We give two coordinatizations, one according to a chief series:
\begin{equation}
\mathfrak{L}\big(A_4\mid (A_4,C_2\times C_2,\{1\}\big)=C_3\wr_{\cL} (C_2\times C_2),
\end{equation}
and another one along a composition series:
\begin{equation}
\mathfrak{L}\big(A_4\mid (A_4,C_2\times C_2,C_2,\{1\}\big)=C_3\wr_{\cL} C_2\wr_{\cL} C_2.
\end{equation}
The first coordinatization admits a nice geometrical interpretation: the top level corresponds to rotations of 3 vertices keeping the other vertex fixed, while the second level represents the possible flips around the 3 diagonals connecting the opposite edges in the tetrahedron. These generate a Klein 4-group $C_2 \times C_2$ acting on these diagonals. Note that the top level order 3 group of rotations maps acts on these diagonals cyclically.

It can be seen that in the second coordinatization two components not in hierarchical relation but completely independent (hence the direct  product  $(C_2\times C_2)$) are forced into a hierarchical structure. However, this is not a limitation of the coordinatization method, as looking at the dependency structure would reveal that there is no real dependency between level 2 and 3, i.e. changes in the 2nd coordinate cannot influence the value of the dependency function on the 3rd level. 
\subsection{ Solving Strategies for the Rubik's Cube}

Each coordinatization (each subgroup chain) corresponds  to a solving strategy of the permutation puzzle. For instance, for the $2\times 2\times 2$ Pocket Cube, the following coordinatization
$$S_8\wr C_3 \wr S_7\wr C_3 \wr S_6\wr C_3 \wr S_5\wr C_3 \wr S_4\wr C_3 \wr S_3\wr C_3 \wr C_2\wr C_3$$
corresponds to  a really step-by-step fashion: get the position and the orientation of the first corner right, then proceed to the next corner until the cube is solved.

Contrasting to the previous, very machine-minded solution, here is another one which is short, and reveals the existence of a different puzzle within the Pocket Cube:
$$S_8 \wr \prod_{i=1}^{7}C_3.$$
The top level component is the right regular representation of the now familiar symmetric group permuting the 8 corners. The second level is the direct product of 7 copies of modulo 3 counters (the orientation group of corners). It is to be noted that not 8 copies, otherwise every corner could be rotated independently from the other corners (and that would be rather easy to solve). Actually solving the bottom level is the same type of problem as the Rubik's Clock \cite{RubiksClock}, which is an array of connected modulo 12 counters. As the underlying group is commutative, it is easier to solve since the order of operations generating 
this subpuzzle does not matter in this lowest level. 

\appendix
\section{}
\begin{prop}
\label{prop:lagrangesize}
  $|G|= |\prod_{1\leq i <  n}G_i/G_{i+1}|$.
\end{prop}

\begin{proof}
 The statement is true for the chain $G>\langle 1\rangle$ which yields the trivial decomposition, as  $G:\langle 1\rangle=|G|$. Now let $H$ and $K$ be consecutive members of a subgroup chain, $H>K$. They contribute in the product by  a factor $H:K=\frac{|H|}{|K|}$ (by Lagrange Theorem). Now we refine the chain by introducing $L$  in between: $H>L>K$. So the contribution is $(H:L)\cdot(L:K)$, which is $\frac{|H|}{|L|}\cdot \frac{|L|}{|K|}=\frac{|H|}{|K|}$.   
\end{proof}

\bibliographystyle{alpha}
\bibliography{sgc}

\end{document}